\documentclass[12pt, reqno]{amsart}
\usepackage{ amsmath,amsthm, amscd, amsfonts, amssymb, graphicx, color}
\usepackage[bookmarksnumbered, colorlinks,linkcolor=blue,anchorcolor=blue,citecolor=blue, plainpages]{hyperref}
\newtheorem{thm}{Theorem}[section]
\newtheorem{cor}[thm]{Corollary}

\newtheorem{exam}[thm]{Example}
\numberwithin{equation}{section}

\begin{document}

\title{additive results of group inverses in Banach algebras}

\author{Dayong Liu\textsuperscript{$\dagger$}}
\address{
\textsuperscript{$\dagger$}College of Science \\ Central South University of Forestry and Technology, China}
\email{<liudy@csuft.edu.cn>}
\author{Huanyin Chen\textsuperscript{$\ddagger$}}
\address{
\textsuperscript{$\ddagger$}School of Mathematics\\ Hangzhou Normal University\\ Hangzhou, China}
\email{<huanyinchen@aliyun.com>}


\begin{abstract}
In this paper, we present new presentations of group inverse for the sum of two group invertible elements in a Banach algebra. We then apply these results to block complex matrices. The group invertibility of certain block complex matrices is thereby obtained.
\end{abstract}

\maketitle

\section{Introduction}

Let $\mathcal{A}$ be a Banach algebra with an identity. An element $a$ in a Banach algebra $\mathcal{A}$ has group inverse provided that there exists $b\in \mathcal{A}$ such that $a=aba, b=bab$ and $ab=ba$. Such $b$ is unique if exists, denoted by $a^{\#}$, and called the group inverse of $a$.
As is well known, a square complex matrix $A$ has group inverse if and only if $rank(A)=rank(A^2)$. The group invertibility in a ring is attractive. It has interesting applications of resistance distances to the bipartiteness of graphs (see~\cite{CE,SW}).
Recently, the group inverse in a Banach algebra or a ring was extensively studied by many authors, e.g., \cite{B,B2,C1,CRV2013,MD,MD2,ZM}. In~\cite[Theorem 2.3]{L}, Liu et al. presented the group inverse of the combinations of two group invertible complex matrices $P$ and $Q$ under the condition $PQQ^{\#}=QPP^{\#}$. In ~\cite[Theorem 3.1]{Z2}, Zhou et al. investigated the group inverse of $a+b$ under the condition $abb^{\#}=baa^{\#}$ in a Dedekind finite ring in which $2$ is invertible. The motivation of this paper is to extend the preceding results to a general setting.

In Section 2, we present the group inverse for the sum of two group invertible elements in a Banach algebra. Let $a,b\in \mathcal{A}^{\#}$. If
$abb^{\#}=\lambda baa^{\#}$, then $a+b\in \mathcal{A}^{\#}$. The representation of its group inverse is also given. In Section 3, we apply our results and investigate the group inverse of a block complex matrix $$M=\left(
\begin{array}{cc}
A&C\\
B&D
\end{array}
\right)$$  where $A\in {\Bbb C}^{m\times m}, B\in {\Bbb C}^{n\times m}, C\in {\Bbb C}^{m\times n}, D\in {\Bbb C}^{n\times n}$. This problem is quite complicated and was expensively studied by many authors. As applications, the group invertibility of certain block complex matrices $M$ is thereby obtained. Additionally, this paper extends the results obtained in ~\cite[Theorem 2.3]{L} and ~\cite[Theorem 3.1]{Z2}.

Throughout the paper, all Banach algebras are complex with an identity. Let $\mathcal{A}$ be the Banach algebra. We use $\mathcal{A}^{-1}$ to denote the set of all invertible elements in $\mathcal{A}$. $\lambda$ always stands for a complex number. ${\Bbb C}^{m\times n}$ stand for the set of all complex $m\times n$ matrices.

\section{main results}

Let $S=\{ e_1,\cdots ,e_n\}$ be a complete set of idempotents in $\mathcal{A}$, i.e., $e_ie_j=0 (i\neq j), e_i^2=e_i (1\leq i\leq n)$ and $\sum\limits_{i=1}^{n}e_i=1$.
Then we have $a=\sum\limits_{i,j=1}^{n}e_iae_j$. We write $a$ as the matrix form
$a=(a_{ij})_S$, where $a_{ij}=e_iae_j\in e_i\mathcal{A}e_j$, and call it the Peirce matrix of $a$ relatively to $S$. We shall use this new technique with relative Peirce matrices and generalize \cite[Theorem 2.3]{L} and ~\cite[Theorem 3.1]{Z2} as follows.

\begin{thm} Let $a,b\in \mathcal{A}^{\#}, \lambda\in {\Bbb C}.$ If $abb^{\#}=\lambda baa^{\#}$, then $a+b\in \mathcal{A}^{\#}$. In this case,
$$(a+b)^{\#}=
\left\{\begin{array}{lll}
(a+b)(a^{\#}+b^{\#})^2&,&\lambda =-1,\\
\frac{1}{1+\lambda}[a^{\#}+b^{\#}-a^{\#}bb^{\#}]+\frac{\lambda}{1+\lambda}[b^{\pi}a^{\#}+a^{\pi}b^{\#}]&,&\lambda\neq -1.
\end{array}\right.$$\end{thm}
\begin{proof} Let $p=aa^{\#}$. Write $$b=\left(
\begin{array}{cc}
b_1&b_2\\
b_3&b_4
\end{array}
\right)_p, bb^{\#}=\left(
\begin{array}{cc}
x_1&x_2\\
x_3&x_4
\end{array}
\right)_p.$$ Then
$$abb^{\#}=\left(
\begin{array}{cc}
ax_1&ax_2\\
0&0
\end{array}
\right)_p, baa^{\#}=\left(
\begin{array}{cc}
b_1&0\\
b_3&0
\end{array}
\right)_p.$$
Since $abb^{\#}=\lambda baa^{\#}$, we have $$ax_1=\lambda b_1, x_2=0, b_3=0.$$ Then $$b=\left(
\begin{array}{cc}
b_1&b_2\\
0&b_4
\end{array}
\right)_p.$$ Moreover, we have $$b^{\#}=\left(
\begin{array}{cc}
b_1^{\#}&z\\
0&b_4^{\#}
\end{array}
\right)_p$$ for some $z\in \mathcal{A}$. This implies that $$x_1=b_1b_1^{\#}, x_3=0, x_4=b_4b_4^{\#}.$$ Therefore $$bb^{\#}=\left(
\begin{array}{cc}
b_1b_1^{\#}&0\\
0&b_4b_4^{\#}
\end{array}
\right)_p.$$
Since $b=(bb^{\#})b=b(bb^{\#})$, we see that $$b_2=(b_1b_1^{\#})b_2=b_2(b_4b_4^{\#}).$$ Then $$b_2=(b_1b_1^{\#})b_2(b_4b_4^{\#}).$$
Let $e_1=b_1b_1^{\#}, e_2=aa^{\#}-b_1b_1^{\#}, e_3=b_4b_4^{\#}, e_4=a^{\pi}-b_4b_4^{\#}$.
Since $a, e_1\in aa^{\#}\mathcal{A}aa^{\#}$, we write $$a=\left(
\begin{array}{cc}
a_1&a_2\\
a_3&a_4
\end{array}
\right)_{e_1}\in aa^{\#}\mathcal{A}aa^{\#}.$$ Then $$\begin{array}{lll}
a_1&=&x_1ax_1=\lambda x_1b_1=\lambda b_1b_1^{\#}b_1=\lambda b_1,\\
a_3&=&(aa^{\#}-e_1)ax_1=ax_1-\lambda b_1=0.\\
\end{array}$$
Moreover, we see that $$a_1, a_4\in (aa^{\#}\mathcal{A}aa^{\#})^{-1}.$$
Since $S=\{ e_1,e_2,e_3,e_4\}$ is a complete set of idempotents in $\mathcal{A}$, we have two Peirce matrices of $a$ and $b$ relatively to $S$:
$$a=\left(
\begin{array}{cccc}
a_1&a_2&0&0\\
0&a_4&0&0\\
0&0&0&0\\
0&0&0&0
\end{array}
\right)_S,
$$
$$b=\left(
\begin{array}{cccc}
b_1&0&b_2&0\\
0&0&0&0\\
0&0&b_4&0\\
0&0&0&0
\end{array}
\right)_S.
$$ Then $$a+b=\left(
\begin{array}{cccc}
(1+\lambda )b_1&a_2&b_2&0\\
0&a_4&0&0\\
0&0&b_4&0\\
0&0&0&0
\end{array}
\right)_S. $$
One directly checks that
$$a^{\#}=\left(
\begin{array}{cccc}
a_1^{-1}&-a_1^{-1}a_2a_4^{-1}&0&0\\
0&a_4^{-1}&0&0\\
0&0&0&0\\
0&0&0&0
\end{array}
\right)_S,$$

$$b^{\#}=\left(
\begin{array}{cccc}
b_1^{-1}  &  0  &  -b_1^{-1}b_2b_4^{-1}   &   0  \\
0         &  0  &  0                      &   0  \\
0         &  0  &  b_4^{-1}               &   0  \\
0         &  0  &  0                      &   0
\end{array}
\right)_S.$$

Then
\begin{align*}
aa^{\#}&= \left(
\begin{array}{cccc}
e_1&0&0&0\\
0&e_2&0&0\\
0&0&0&0\\
0&0&0&0
\end{array}
\right)_S,\ \ \
bb^{\#}&=\left(
\begin{array}{cccc}
e_1&0&0&0\\
0&0&0&0\\
0&0&e_3&0\\
0&0&0&0
\end{array}
\right)_S.
\end{align*}
Case 1. $\lambda=-1$. Then
\begin{align*}
a+b&=\left(
\begin{array}{cccc}
0&a_2&b_2&0\\
0&a_4&0&0\\
0&0&b_4&0\\
0&0&0&0
\end{array}
\right)_S,
\end{align*}
\begin{align*}
(a+b)^{\#}&=\left(
\begin{array}{cccc}
0&a_2(a_4)^{-2}&b_2(b_4)^{-2}&0\\
0&(a_4)^{-1}&0&0\\
0&0&(b_4)^{-1}&0\\
0&0&0&0
\end{array}
\right)_S.
\end{align*}
Since $a_1^{-1}+b_1^{-1}=a^{-1}(b_1+a_1)b_1^{-1}=a^{-1}(1+\lambda )b_1b_1^{-1}=0$, we see that
\begin{align*}
a^{\#}+b^{\#}=\left(
\begin{array}{cccc}
0&-a_1^{-1}a_2a_4^{-1}&-b_1^{-1}b_2b_4^{-1}&0\\
0&a_4^{-1}&0&0\\
0&0&b_4^{-1}&0\\
0&0&0&0
\end{array}
\right)_S.
\end{align*}
Therefore
\begin{align*}
&(a+b)(a^{\#}+b^{\#})^2    \\
&=\left(
\begin{array}{cccc}
0&a_2&b_2&0\\
0&a_4&0&0\\
0&0&b_4&0\\
0&0&0&0
\end{array}
\right)_S\left(
\begin{array}{cccc}
0&-a_1^{-1}a_2a_4^{-1}&-b_1^{-1}b_2b_4^{-1}&0\\
0&a_4^{-1}&0&0\\
0&0&b_4^{-1}&0\\
0&0&0&0
\end{array}
\right)_S^2   \\
&=(a+b)^{\#},
\end{align*}
 as desired.

Case 2. $\lambda \neq -1$. Then
\begin{align*}
&(a+b)^{\#}\\
&=\left(
\begin{array}{cccc}
(1+\lambda )^{-1}b_1^{-1}&-(1+\lambda )^{-1}b_1^{-1}a_2a_4^{-1}&-(1+\lambda )^{-1}b_1^{-1}b_2b_4^{-1}&0\\
0&a_4^{-1}&0&0\\
0&0&b_4^{-1}&0\\
0&0&0&0
\end{array}
\right)_S\\
&=(1+\lambda )^{-1}\left(
\begin{array}{cccc}
b_1^{-1}&-b_1^{-1}a_2a_4^{-1}&-b_1^{-1}b_2b_4^{-1}&0\\
0&0&0&0\\
0&0&0&0\\
0&0&0&0
\end{array}
\right)_S\\
&\quad +\left(
\begin{array}{cccc}
0&0&0&0\\
0&a_4^{-1}&0&0\\
0&0&0&0\\
0&0&0&0
\end{array}
\right)_S+\left(
\begin{array}{cccc}
0&0&0&0\\
0&0&0&0\\
0&0&b_4^{-1}&0\\
0&0&0&0
\end{array}
\right)_S.
\end{align*}
We compute that
\begin{align*}
&  a^{\#}+b^{\#}-a^{\#}bb^{\#}   \\
&=\left(
\begin{array}{cccc}
a_1^{-1}&-a_1^{-1}a_2a_4^{-1}&0&0\\
0&a_4^{-1}&0&0\\
0&0&0&0\\
0&0&0&0
\end{array}
\right)_S+\left(
\begin{array}{cccc}
b_1^{-1}&0&-b_1^{-1}b_2b_4^{-1}&0\\
0&0&0&0\\
0&0&b_4^{-1}&0\\
0&0&0&0
\end{array}
\right)_S  \\
&\quad -\left(
\begin{array}{cccc}
a_1^{-1}&0&0&0\\
0&0&0&0\\
0&0&0&0\\
0&0&0&0
\end{array}
\right)_S
\end{align*}
\quad $=\left(
\begin{array}{cccc}
b_1^{-1}&-a_1^{-1}a_2a_4^{-1}&-b_1^{-1}b_2b_4^{-1}&0\\
0&a_4^{-1}&0&0\\
0&0&b_4^{-1}&0\\
0&0&0&0
\end{array}
\right)_S,$  \\
\begin{align*}
b^{\pi}a^{\#}&=\left(
\begin{array}{cccc}
0&0&0&0\\
0&a_4^{-1}&0&0\\
0&0&0&0\\
0&0&0&0
\end{array}
\right)_S, \ \
a^{\pi}b^{\#}&=\left(
\begin{array}{cccc}
0&0&0&0\\
0&0&0&0\\
0&0&b_4^{-1}&0\\
0&0&0&0
\end{array}
\right)_S.
\end{align*}
Therefore we have
\begin{align*}
&\displaystyle\frac{1}{1+\lambda}[a^{\#}+b^{\#}-a^{\#}bb^{\#}]+\frac{\lambda}{1+\lambda}[b^{\pi}a^{\#}+a^{\pi}b^{\#}]\\
&=(1+\lambda )^{-1}\left(
\begin{array}{cccc}
b_1^{-1}&-b_1^{-1}a_2a_4^{-1}&-b_1^{-1}b_2b_4^{-1}&0\\
0&0&0&0\\
0&0&0&0\\
0&0&0&0
\end{array}
\right)_S\\
& \quad +\left(
\begin{array}{cccc}
0&0&0&0\\
0&a_4^{-1}&0&0\\
0&0&0&0\\
0&0&0&0
\end{array}
\right)_S+\left(
\begin{array}{cccc}
0&0&0&0\\
0&0&0&0\\
0&0&b_4^{-1}&0\\
0&0&0&0
\end{array}
\right)_S.
\end{align*} Therefore we have $$(a+b)^{\#}=\frac{1}{1+\lambda}[a^{\#}+b^{\#}-a^{\#}bb^{\#}]+\frac{\lambda}{1+\lambda}[b^{\pi}a^{\#}+a^{\pi}b^{\#}],$$ as asserted.\end{proof}

\begin{cor} Let $a,b\in \mathcal{A}^{\#}, \lambda\in {\Bbb C}.$ If $aa^{\#}b=\lambda bb^{\#}a$, then $a+b\in \mathcal{A}^{\#}$. In this case,
$$(a+b)^{\#}=
\left\{\begin{array}{lll}
(a^{\#}+b^{\#})^2(a+b)&,&\lambda =-1,\\
\frac{1}{1+\lambda}[a^{\#}+b^{\#}-aa^{\#}b^{\#}]+\frac{\lambda}{1+\lambda}[a^{\#}b^{\pi}+b^{\#}a^{\pi}]&,&\lambda\neq -1.
\end{array}\right.$$\end{cor}
\begin{proof} Let $(R,*)$ be the opposite ring of $R$. That is, it is a ring with the multiplication $a*b=b\cdot a$.
Applying Theorem 2.1 to the opposite ring $(R,*)$ of $R$, we obtain the result.\end{proof}

\begin{cor} Let $a,b\in \mathcal{A}$ be idempotents, and let $\lambda\in {\Bbb C}.$ If $ab=\lambda ba$, then $a+b\in \mathcal{A}^{\#}$. In this case,
\begin{align*}
(a+b)^{\#}&=\left\{\begin{array}{lll}
(a+b)^3&,&\lambda =-1,\\
\displaystyle a+b-\frac{2+\lambda}{1+\lambda}ab&,&\lambda\neq -1.
\end{array}\right.
\end{align*}\end{cor}
\begin{proof} Since $a$ and $b$ are idempotents, we have
$$aa^{\#}b=ab=\lambda ba=\lambda bb^{\#}a.$$ Therefore we establish the result by Corollary 2.2.\end{proof}

\begin{thm} Let $a,b\in \mathcal{A}^{\#}, \lambda\in {\Bbb C}.$ If $abb^{\#}=\lambda b (\lambda \neq -1)$, then $a+b\in \mathcal{A}^{\#}$. In this case,
\begin{align*}
(a+b)^{\#}&=(1+\lambda )^{-1}b^{\#}+b^{\pi}a^{\#}b^{\pi}\\
&\quad +\lambda (1+\lambda )^{-2}b^{\#}aa^{\#}b^{\pi}-(1+\lambda )^{-1}b^{\#}ab^{\pi}a^{\#}b^{\pi}.
\end{align*}\end{thm}
\begin{proof} Let $p=bb^{\#}$. Write $a=\left(
\begin{array}{cc}
a_1&a_2\\
a_3&a_4
\end{array}
\right)_p, b=\left(
\begin{array}{cc}
b&0\\
0&0
\end{array}
\right)_p$. Then $a_1=bb^{\#}abb^{\#}=\lambda bb^{\#}b=\lambda b$ and $a_3=(1-bb^{\#})abb^{\#}=\lambda (1-bb^{\#})b=0$. Hence,
$$a=\left(
\begin{array}{cc}
\lambda b&a_2\\
0&a_4
\end{array}
\right)_p, \ \ \  a+b=\left(
\begin{array}{cc}
(1+\lambda )b&a_2\\
0&a_4
\end{array}
\right)_p.$$
Obviously, $a_4=(1-bb^{\#})a(1-bb^{\#})=b^{\pi}a$.
We easily check that $a_4^{\#}=b^{\pi}a^{\#}b^{\pi}$ and $a_4^{\pi}=1-b^{\pi}ab^{\pi}a^{\#}b^{\pi}=1-b^{\pi}aa^{\#}b^{\pi}.$
Moreover, we have
\begin{align*}
[(1+\lambda )b]^{\pi}a_2a_4^{\pi}&=b^{\pi}bb^{\#}a_4^{\#}\\
&=b^{\pi}bb^{\#}b^{\pi}a^{\#}b^{\pi}\\
&=0.
\end{align*}
 According to~\cite[Theorem 2.1]{MD}, $a+b\in \mathcal{A}^{\#}.$
Further, we have $$(a+b)^{\#}=\left(
\begin{array}{cc}
[(1+\lambda )b]^{\#}&z\\
0&a_4^{\#}
\end{array}
\right)_p,$$ where $z=[(1+\lambda )^{-1}b^{\#}]^2a_2a_4^{\pi}-(1+\lambda )^{-1}b^{\#}a_2a_4^{\#}.$
We compute that
\begin{align*}
a_2a_4^{\pi}&=bb^{\#}ab^{\pi}[1-b^{\pi}aa^{\#}b^{\pi}]\\
&=bb^{\#}ab^{\pi}a^{\pi}b^{\pi}]\\
&=bb^{\#}abb^{\#}a^{\pi}b^{\pi}]\\
&=\lambda baa^{\#}b^{\pi}.
\end{align*}
Therefore, we have
\begin{align*}
(a+b)^{\#}&=(1+\lambda )^{-1}b^{\#}+b^{\pi}a^{\#}b^{\pi}\\
          & \quad +[(1+\lambda )^{-1}b^{\#}]^2a_2a_4^{\pi}-(1+\lambda )^{-1}b^{\#}a_2a_4^{\#}\\
          &=(1+\lambda )^{-1}b^{\#}+b^{\pi}a^{\#}b^{\pi}\\
          & \quad +\lambda (1+\lambda )^{-2}b^{\#}aa^{\#}b^{\pi}-(1+\lambda )^{-1}b^{\#}ab^{\pi}a^{\#}b^{\pi},
\end{align*} as asserted.\end{proof}

\begin{cor} Let $a,b\in \mathcal{A}^{\#}, \lambda\in {\Bbb C}.$ If $aa^{\#}b=\lambda a (\lambda \neq -1)$, then $a+b\in \mathcal{A}^{\#}$. In this case,
\begin{align*}
(a+b)^{\#}&=(1+\lambda )^{-1}a^{\#}+a^{\pi}b^{\#}a^{\pi}\\
&\quad +\lambda (1+\lambda )^{-2}a^{\pi}bb^{\#}a^{\#}-(1+\lambda )^{-1}a^{\pi}b^{\#}a^{\pi}ba^{\#}.
\end{align*}
\end{cor}
\begin{proof} Let $(\mathcal{A},*)$ be the opposite algebra of $\mathcal{A}$. Apply Theorem 2.4 to elements $b,a$ in this opposite ring, we obtain the result.\end{proof}

We demonstrate Theorem 2.4 by the following numerical example.

\begin{exam} Let $A=\left(
  \begin{array}{cc}
    -1&-1\\
    1&-3
  \end{array}
\right), B=\left(
\begin{array}{ccc}
0&1\\
0&1
\end{array}
\right)\in {\Bbb C}^{2\times 2}$. Then $A$ and $B$ have group inverses and $ABB^{\#}=-2B$. Since
$A^{\#}=\left(
  \begin{array}{cc}
    -\frac{3}{4}&\frac{1}{4}\\
    -\frac{1}{4}&-\frac{1}{4}
  \end{array}\right)$ and $B^{\#}=\left(
\begin{array}{ccc}
0&1\\
0&1
\end{array}
\right)$. By using Theorem 2.4, we get
\begin{align*}
(A+B)^{\#}&=-B^{\#}+B^{\pi}A^{\#}B^{\pi}-2B^{\#}AA^{\#}B^{\pi}+B^{\#}AB^{\pi}A^{\#}B^{\pi}\\
&=\left(
  \begin{array}{cc}
    -1&0\\
    -\frac{1}{2}&-\frac{1}{2}
  \end{array}\right).
  \end{align*}\end{exam}

\section{applications}

The aim of this section is to present the group invertibility of the block matrix $M$ by using our main results. We are ready to prove:

\begin{thm} Let $A$ and $D$ have group inverses. If $A^{\pi}B=0, D^{\pi}C=0, ACD^{\#}=\lambda C$ and $BCD^{\#}=\lambda D,$ then $M$ has group inverse.\end{thm}
\begin{proof} Write $M=P+Q$, where $$P=\left(
  \begin{array}{cc}
    A&0\\
    B&0
  \end{array}
\right), Q=\left(
  \begin{array}{cc}
    0&C\\
    0&D
  \end{array}
\right).$$
Since $A^{\pi}B=0, D^{\pi}C=0$, it follows by~\cite[Theorem 3.4]{B} that $P$ and $Q$ have group inverses.
Moreover, we get
$$Q^{\#}=\left(
  \begin{array}{cc}
    0&C(D^{\#})^2\\
    0&D^{\#}
  \end{array}
\right).$$
We easily check that
\begin{align*}
PQQ^{\#}&=\left(
  \begin{array}{cc}
    A&0\\
    B&0
  \end{array}
\right)\left(
  \begin{array}{cc}
    0&C\\
    0&D
  \end{array}
\right)\left(
  \begin{array}{cc}
    0&C(D^{\#})^2\\
    0&D^{\#}
  \end{array}
\right)\\
&=\left(
  \begin{array}{cc}
    A&0\\
    B&0
  \end{array}
\right)\left(
  \begin{array}{cc}
    0&CD^{\#}\\
    0&D^{\#}
  \end{array}
\right)\\
&=\left(
  \begin{array}{cc}
    0&ACD^{\#}\\
    0&BCD^{\#}
  \end{array}
\right)\\
&=\lambda\left(
  \begin{array}{cc}
    0&C\\
    0&D
  \end{array}
\right)\\
&=\lambda Q.
\end{align*}
In light of Theorem 2.4, $M=P+Q$ has group inverse, as desired.\end{proof}

\begin{cor} Let $A$ and $D$ have group inverses. If $CD^{\pi}=0, BA^{\pi}=0, A^{\#}BD=\lambda B$ and $A^{\#}BC=\lambda A,$, then $M$ has group inverse.\end{cor}
\begin{proof} Applying Theorem 3.1 to the block matrix $$M^T=\left(
  \begin{array}{cc}
    D^T&B^T\\
    C^T&A^T
  \end{array}
\right),$$ we prove that $M^T$ has group inverse. Therefore, we easily check that
$M=(M^T)^T$ has group inverse, as asserted.\end{proof}

\begin{thm} Let $A$ and $D$ have group inverses. If $A^{\pi}C=0, D^{\pi}B=0, A^{\#}AB=\lambda A$ and $A^{\#}AD=\lambda C,$ then $M$ has group inverse.\end{thm}
\begin{proof} Write $M=P+Q$, where $$P=\left(
  \begin{array}{cc}
    A&C\\
    0&0
  \end{array}
\right),Q=\left(
  \begin{array}{cc}
    0&0\\
    B&D
  \end{array}
\right).$$
Since $A^{\pi}C=0, D^{\pi}B=0$, by using ~\cite[Theorem 3.4]{B} that $P$ and $Q$ have group inverses.
Moreover, we get
$$P^{\#}=\left(
  \begin{array}{cc}
    A^{\#}&(A^{\#})^2A\\
    0&0
  \end{array}
\right).$$
Then we have
\begin{align*}
PP^{\#}Q&=\left(
  \begin{array}{cc}
    A&C\\
    0&0
  \end{array}
\right)\left(
  \begin{array}{cc}
    A^{\#}&(A^{\#})^2A\\
    0&0
  \end{array}
\right)\left(
  \begin{array}{cc}
    0&0\\
    B&D
  \end{array}
\right)\\
&=\left(
  \begin{array}{cc}
    AA^{\#}&A^{\#}A\\
    0&0
  \end{array}
\right)\left(
  \begin{array}{cc}
    0&0\\
    B&D
  \end{array}
\right)\\
&=\left(
  \begin{array}{cc}
    A^{\#}AB&A^{\#}AD\\
    0&0
  \end{array}
\right)\\
&=\lambda \left(
  \begin{array}{cc}
    A&C\\
    0&0
  \end{array}
\right)\\
&=\lambda P.
\end{align*}
In light of Theorem 2.4, $M=P+Q$ has group inverse, as desired.\end{proof}

\begin{cor} Let $A$ and $D$ have group inverses. If $BD^{\pi}=0, CA^{\pi}=0, CDD^{\#}=\lambda D$ and $ADD^{\#}=\lambda B,$ then $M$ has group inverse.\end{cor}
\begin{proof} Applying Theorem 3.3 to the block matrix $$M^T=\left(
  \begin{array}{cc}
    D^T&B^T\\
    C^T&A^T
  \end{array}
\right),$$ we easily obtain the result as in Corollary 3.2.\end{proof}

It is convenient at this stage to prove the following.

\begin{thm} Let $A\in {\Bbb C}^{m\times m}, D\in {\Bbb C}^{n\times n}$ be idempotents and $rank(B)=rank(C)=rank(BC)=rank(CB)$. If $AD=\lambda AC, A$ $(I-CB)=0$ and $DBA^{\pi}C=0$, then $M$ has group inverse.
\end{thm}
\begin{proof} Since $r(B)=r(C)=r(BC)=r(CB)$, it follows by ~\cite[Lemma 2.3]{BZ} that $BC$ and $CB$ have group inverses.
Let $K=\left(
  \begin{array}{cc}
   0&C \\
   B&0
  \end{array}
\right)$. Then $K^2=\left(
  \begin{array}{cc}
   CB&0 \\
   0&BC
  \end{array}
\right).$ By hypothesis, we have
\begin{align*}
rank(K^2)&=rank(CB)+rank(BC)\\
&=rank(C)+rank(B)\\
&=rank(K).
\end{align*} Then $K$ has group inverse.  \\
\indent
Write $Q:=\left(
  \begin{array}{cc}
   0&A^{\pi}C\\
   D^{\pi}B&0
   \end{array}
\right).$ Then we have $$Q=\left(
  \begin{array}{cc}
    A^{\pi}&0 \\
    0&D^{\pi}
  \end{array}
\right)\left(
  \begin{array}{cc}
   0&C\\
   B&0
   \end{array}
\right).$$ By hypothesis, we see that
$$Q=\left(
  \begin{array}{cc}
   0&C\\
   B&0
   \end{array}
\right)\left(
  \begin{array}{cc}
    A^{\pi}&0 \\
    0&D^{\pi}
  \end{array}
\right).$$ Therefore $N$ has group inverse and $$Q^{\#}=\left(
  \begin{array}{cc}
    A^{\pi}&0 \\
    0&D^{\pi}
  \end{array}
\right)\left(
  \begin{array}{cc}
  0& C(BC)^{\#} \\
  B(CB)^{\#}&0
  \end{array}
\right).$$
Let $P=\left(
  \begin{array}{cc}
    A & AC \\
    DB&D
  \end{array}
\right)$. Then $M=P+Q$. Clearly, $A^{\#}A(DB)=ADB=\lambda A, A^{\#}AD=AD=\lambda AC$, $A^{\pi}(AC)=0$ and $D^{\pi}(DB)=0$. In light of Theorem 3.3,
$P$ has group inverse. Since $ACD^{\pi}B=0, DBA^{\pi}C=0$, we check that
\begin{align*}
PQ&=\left(
  \begin{array}{cc}
    A & AC \\
    DB&D
  \end{array}
\right)\left(
  \begin{array}{cc}
   0&A^{\pi}C\\
   D^{\pi}B&0
   \end{array}
\right)\\
&=0.
\end{align*}
According to ~\cite[Theorem 2.1]{B}, $M$ has group inverse, as asserted.\end{proof}

\begin{cor} Let $A\in {\Bbb C}^{m\times m},D\in {\Bbb C}^{n\times n}$ be idempotents and $rank(B)=rank(C)=rank(BC)=rank(CB)$. If $AD=\lambda BD, (I-CB)D=0$ and $BD^{\pi}CA=0$, then $M$ has group inverse.
\end{cor}
\begin{proof} Applying Theorem 3.5 to the block matrix $$M^T=\left(
  \begin{array}{cc}
    D^T&B^T\\
    C^T&A^T
  \end{array}
\right),$$ we complete the proof as in Corollary 3.4.\end{proof}

\section{Disclosure statement}
Authors declare that there have no potential conflict of interest to report.


\end{document}